\documentclass[12pt]{article}
\usepackage{amssymb,amsmath,amsthm,amsfonts,enumerate, bbm, mathdots}

\usepackage{latexsym}
\usepackage{amscd}
\usepackage{stmaryrd}
\usepackage{color}
\usepackage[all]{xypic}
\usepackage{epsfig}
\usepackage{graphics}
\usepackage{ifthen}
\usepackage{varioref}
\usepackage{rotating}
\usepackage{extarrows}
\usepackage{cite}
\usepackage{mathrsfs}

\numberwithin{equation}{section}

\theoremstyle{plain}
\newtheorem{thm}{Theorem}[section]
\newtheorem{cor}[thm]{Corollary}
\newtheorem{lem}[thm]{Lemma}
\newtheorem{prop}[thm]{Proposition}
\newtheorem{rem}[thm]{Remark}

\newmuskip\pFqmuskip
\newcommand*\pFq[6][8]{%
 \begingroup % only local assignments
 \pFqmuskip=#1mu\relax
 \mathchardef\normalcomma=\mathcode`,
 \mathcode`\,=\string"8000
 \begingroup\lccode`\~=`\,
 \lowercase{\endgroup\let~}\pFqcomma
 {}_{#2}F_{#3}{\left(\genfrac..{0pt}{}{#4}{#5};#6\right)}%
 \endgroup
}
\newcommand{\pFqcomma}{{\normalcomma}\mskip\pFqmuskip}

\makeatletter
\newenvironment{proofof}[1]{\par
 \pushQED{\qed}%
 \normalfont \topsep6\p@\@plus6\p@\relax
 \trivlist
 \item[\hskip\labelsep
    \bfseries
  Proof of #1\@addpunct{.}]\ignorespaces
}{%
 \popQED\endtrivlist\@endpefalse
}
\makeatother

\definecolor{dg}{rgb}{0.0625,0.64,0.0625}
\usepackage[%dvipdfm,
      pdfstartview=FitH,
      CJKbookmarks=true,
      bookmarksnumbered=true,
      bookmarksopen=true,
      colorlinks,
      pdfborder=001,
      linkcolor=blue,
      anchorcolor=green,
      citecolor=red
      ]{hyperref}

\definecolor{deepmaroon}{RGB}{100,0,20}   % 深暗红
\definecolor{deepnavy}{RGB}{0,30,80}    % 藏青深蓝
\definecolor{deepforest}{RGB}{0,60,30}   % 森林深绿
\definecolor{deeppurple}{RGB}{128,0,255}   % 暗紫
\definecolor{deepteal}{RGB}{0,70,70}    % 深湖蓝
\definecolor{deepolive}{RGB}{60,65,30}   % 暗橄榄绿
\definecolor{darkredheavy}{RGB}{120,0,0}   % 浓深红
\definecolor{lightgreen}{RGB}{0,255,255}  % 纯浓藏青
\definecolor{darkgreenheavy}{RGB}{0,80,0}  % 浓墨绿
\definecolor{darkpurpleheavy}{RGB}{90,0,120} % 浓深紫
\definecolor{darktealheavy}{RGB}{0,90,90}  % 浓深青
\definecolor{darkorangeheavy}{RGB}{160,60,0} % 深橘棕
\definecolor{matcha}{RGB}{130,190,140}
\definecolor{emeraldDeep}{RGB}{0,115,90}
\definecolor{grassBright}{RGB}{20,180,60}

\newfont{\scyr}{wncyr10 scaled 550}

\def\proof{\noindent {\bf Proof.\;}}

\def\wt{\operatorname{wt}}
\def\dep{\operatorname{dep}}
\def\height{\operatorname{ht}}

\allowdisplaybreaks

\newcommand{\A}{\mathcal{A}}

\newcommand{\Z}{\mathbb{Z}}

\begin{document}

\title{Refined sum formulas for finite multiple zeta values of level two}
\date{~}
\author{
Zhenlu Wang\\\small
School of Artificial Intelligence, Taizhou University,\\ \small Taizhou 318000, Zhejiang, China\\\small
\textrm{zhenluwang@tzc.edu.cn}
}

\maketitle

\begin{abstract}
  In this paper, we establish two classes of refined sum formulas for finite multiple zeta values of level two, a variant of finite multiple zeta values recently introduced by M. Kaneko et al. One is an explicit formula for the sum of level-two finite multiple zeta values over indices in which one prescribed component is odd and all the remaining components are even. The other is an Ohno--Zagier type relation for level-two finite multiple zeta values, which gives a generating function for sums of fixed weight, depth and height in terms of depth-one values. As applications, we derive sum formulas for fixed weight and depth and for several special heights.
\end{abstract}

{\small

}

%%---------------------------------------------------------------------------
%%------------------------Content-------------------------------------------
%%----------------------------------------------------------------------------

\section{Introduction}
A finite sequence of positive integers is called an index. The weight, depth and height of an index $\boldsymbol{k}=(k_1,\ldots,k_r)$ are defined by
\begin{align*}
\wt(\boldsymbol{k})=k_1+\cdots+k_r,\quad
\dep(\boldsymbol{k})=r,\quad
\height(\boldsymbol{k})=\#\{j|1\le j\le r,k_j\ge 2\},
\end{align*}
respectively. Let $\mathcal P$ be the set of prime numbers and define the $\mathbb Q$-algebra $\mathcal{A}$ by
\begin{align*}
\A=\prod_{p\in\mathcal P}\Z/p\Z\Big/\bigoplus_{p\in\mathcal P}\Z/p\Z.
\end{align*}
Following Kaneko \cite{Kaneko2019}, for an index $\boldsymbol{k}=(k_1,\ldots,k_r)$, the finite multiple zeta value (FMZV) is defined by
\begin{align*}
\zeta_{\A}(\boldsymbol{k})=\left(\sum_{0<n_1<\cdots<n_r<p}\frac{1}{n_1^{k_1}\cdots n_r^{k_r}}\pmod p\right)_{p}\in\A.
\end{align*}
This is a finite analogue of the classical multiple
zeta value (MZV), defined by
\begin{align*}
\zeta(\boldsymbol{k})=\zeta(k_1,\ldots,k_r) =\sum_{0<n_1<\cdots<n_r}\frac{1}{n_1^{k_1}\cdots n_r^{k_r}}.
\end{align*}
This series converges if and only if $k_r\ge 2$, whereas no such restriction is required for the finite multiple zeta value $\zeta_{\A}(\boldsymbol{k})$. 

Multiple zeta values were systematically studied by Hoffman \cite{Hoffman1992} and Zagier \cite{Zagier1994} independently in the early 1990s and have since played an important role in both mathematics and theoretical physics. Since then, a large amount of work has been devoted to multiple zeta values and their numerous variants and generalizations, including finite multiple zeta values; see Zhao's monograph \cite{Zhao2016} for further details.

Kaneko and Zagier\cite{Kaneko2019,KanekoZagier2000} proposed a deep conjecture relating finite multiple zeta values to classical multiple zeta values through a mysterious connection. Although the conjecture remains open, numerous parallel results have been established on both the finite and classical sides, including sum formulas and weighted sum formulas for finite multiple zeta values \cite{HiroseMuraharaSaito2019,Kaneko2019,Kamano2018,SaitoWakabayashi2015}.

Recently, level-two variants obtained by imposing parity restrictions on the summation variables have attracted considerable attention. Important classical examples include Hoffman's multiple $t$-values \cite{Hoffman2019} and Kaneko--Tsumura's multiple $T$-values \cite{KanekoTsumura2020}. For an index $\boldsymbol{k}=(k_1,\ldots,k_r)$, Kaneko, Murakami and Yoshihara \cite{KanekoMurakamiYoshihara2023} introduced the finite multiple zeta value of level two (level-two FMZV) by
\begin{align*}
\zeta_{\A}^{(2)}(\boldsymbol{k})=\left(\sum_{0<n_1<\cdots<n_r<p/2}\frac{1}{n_1^{k_1}\cdots n_r^{k_r}}\pmod p\right)_p\in\A.
\end{align*}
Unlike ordinary finite multiple zeta values, the summation here is truncated at $p/2$ rather than at $p$. In particular,  level-two FMZVs may also be viewed, up to a constant multiple, as finite analogues of multiple $t$-values. They also showed \cite[Proposition 2.1]{KanekoMurakamiYoshihara2023} that 
\begin{align}
\zeta_{\A}^{(2)}(2m)=0\qquad(m\ge1).\label{eq:evenzero}
\end{align} 
Moreover, they established parity results and several sum formulas, and concluded with a weighted sum conjecture, which was recently proved by Li, the present author, and Zhang \cite{LiWangZhang2026}. Zhao \cite{ZhaoAxioms2024,ZhaoFoundations2024} studied finite multiple $T$-values and subsequently introduced finite multiple mixed values, which provide a common generalization of several level-two variants of FMZVs. More recently, Li and the present author \cite{LiWang2026} obtained a weighted sum formula for finite multiple mixed values.

The purpose of this paper is to establish two classes of refined sum formulas for  level-two FMZVs. The first concerns the following positional parity-restricted sum:
\begin{align*}
\sum_{\substack{k_1+\cdots+k_r=k\\k_i:\ {\rm odd},\ k_j:\ {\rm even}\ (\forall j\ne i)}}\zeta_{\A}^{(2)}(k_1,\ldots,k_r).
\end{align*}
Kaneko et al. proved that the above sum is a rational multiple of $\zeta_{\A}^{(2)}(k)$, but did not provide a general closed form for the rational coefficient \cite[Theorem 3.3]{KanekoMurakamiYoshihara2023}. We determine this rational coefficient explicitly and obtain the following theorem. 
\begin{thm}\label{thm:parity-sum}
Let $r$ be a positive integer and let $k\ge2r-1$ be an odd integer. For a fixed $i$ with $1\le i\le r$, we have
\begin{align}
\sum_{\substack{k_1+\cdots+k_r=k\\k_i:\ {\rm odd},\ k_j:\ {\rm even}\ (\forall j\ne i)}}\zeta_{\A}^{(2)}(k_1,\ldots,k_r)=C_{k,r,i}\zeta_{\A}^{(2)}(k),\label{eq:mainparity}
\end{align}
where $s_i=\min\{i,r-i\}$ and 
\begin{align*}
C_{k,r,i}=\frac{(-1)^{r-1}}{4^{r-1}}\binom{k-r}{r-1}\binom{r-1}{i-1}\prod_{j=0}^{s_i-1}\frac{k-2j}{k-2r+2+2j}.
\end{align*}
\end{thm}
Here and below, an empty product is understood to be $1$, and $\{ k\}^a$ denotes the sequence of $k$ repeated $a$ times. The case $r=2$ of \eqref{eq:mainparity} also follows from \cite[Proposition 2.1(ii)]{KanekoMurakamiYoshihara2023}. The proof of Theorem \ref{thm:parity-sum} is based on the harmonic product for level-two FMZVs. For example, one has 
\begin{align*}
\zeta_{\A}^{(2)}(a)\zeta_{\A}^{(2)}(b)=\zeta_{\A}^{(2)}(a,b)+\zeta_{\A}^{(2)}(b,a)+\zeta_{\A}^{(2)}(a+b).
\end{align*}
Combining the harmonic product with \eqref{eq:evenzero}, we derive a recurrence relation for the rational coefficients $C_{k,r,i}$. Using the evaluation of $ \zeta_{\A}^{(2)}(\{2\}^{a},1,\{2\}^{b})$ due to Hessami Pilehrood, Hessami Pilehrood, and Tauraso \cite{HPT2014}, we obtain the desired explicit expression for $C_{k,r,i}$. 

For integers $k\ge r\ge 1$ and $ 0\le h\le r$, let $I(k, r, h)$ be the set of indices of weight $k$, depth $r$, and height $h$. The second concerns the following sum for level-two FMZVs of fixed weight, depth and height:
\begin{align*}
S(k,r,h)=\sum_{\boldsymbol{k}\in I(k,r,h)}\zeta_{\A}^{(2)}(k_1,\ldots,k_r).%\quad k\ge r\ge 1,\quad 0\le h\le r.
\end{align*}

%If the indices set is empty, the sum is treated as zero. We also set S^{(2)}(0,0,0)=1.

\begin{thm}\label{thm:oz-sum}
Let $x,y,u$ be formal variables. In $\A[[x,y,u]]$, we have
\begin{align}
1+\sum_{k\ge r\ge1}\sum_{h=0}^{r}S(k,r,h)x^{k-r}y^ru^h=\exp\left\{\sum_{n\ge1}\frac{\zeta_{\A}^{(2)}(n)}{n}\left(x^n-\alpha^n-\beta^n\right)\right\},\label{eq:OZ}
\end{align}
where $\alpha$ and $\beta$ are formally determined by
\begin{align}
\alpha+\beta=x-y,\qquad\alpha\beta=-xy(1-u).\label{eq:alphabeta}
\end{align}
\end{thm}
\begin{rem}
Since only the symmetric combinations $\alpha^n+\beta^n$ occur in \eqref{eq:OZ}, they can be expressed in terms of $\alpha+\beta$ and $\alpha\beta$. Hence the right-hand side of \eqref{eq:OZ} is well defined in $\mathcal A[[x,y,u]]$.
\end{rem}
 Theorem \ref{thm:oz-sum} provides a finite level-two analogue of the Ohno--Zagier relation for classical MZVs, which gives a generating function for sums of multiple zeta values of fixed weight, depth and height in terms of Riemann zeta values \cite{OhnoZagier2001}. As applications, we derive sum formulas for fixed weight and depth, height zero, height one, and maximal height, together with an identity for indices of the form $(\{2\}^{a},m,\{2\}^{b})$.

The paper is organized as follows. In Section \ref{sec:parity-sum}, we prove Theorem \ref{thm:parity-sum}. In Section \ref{sec:oz-sum}, we establish the Ohno--Zagier type relation and give its principal corollaries, including sum formulas for fixed weight and depth and for several special heights.

%%---------------------------%%--------------------
\section{Positional parity-restricted sum}\label{sec:parity-sum}
For an odd integer $k$ and integers $r,i$ satisfying 
$$1\le r\le\frac{k+1}{2},\qquad 1\le i\le r,$$ 
we denote the positional parity-restricted sum by $P_{k,r}^{(i)}$ as follows: 
\begin{align*}
 P_{k,r}^{(i)}=\sum_{\substack{k_1+\cdots+k_r=k\\k_i:\ {\rm odd},\ k_j:\ {\rm even}\ (\forall j\ne i)}}\zeta_{\A}^{(2)}(k_1,\ldots,k_r).%\label{eq:Pdef}
\end{align*}

Combining the harmonic product with \eqref{eq:evenzero}, we obtain the following recurrence relation for the positional parity-restricted sums $P_{k,r}^{(i)}$.

\begin{lem}\label{lem:posrec}
Let $k$ be an odd integer and suppose that
$$2\le r\le\frac{k+1}{2},\qquad 1\le i\le r-1.$$
Then we have
\begin{align}
(r-i)P_{k,r}^{(i)}+iP_{k,r}^{(i+1)}+\frac{k-2r+3}{2}P_{k,r-1}^{(i)}=0.\label{eq:posrec}
\end{align}
\end{lem}
\proof
By \eqref{eq:evenzero}, we have
\begin{align}
\sum_{l=1}^{(k-2r+3)/2}\zeta_{\A}^{(2)}(2l)P_{k-2l,r-1}^{(i)}=0.\label{eq:zero-product}
\end{align}
Expanding the left-hand side of \eqref{eq:zero-product} by the harmonic product, each resulting term appears in one of the sums $$P_{k,r}^{(i)},\quad P_{k,r}^{(i+1)},\quad P_{k,r-1}^{(i)}.$$ 

For the insertion terms, if the odd component of the resulting index is in position $i$, then deleting one of the $r-i$ even components to its right recovers an index contributing to $P_{k-2l,r-1}^{(i)}$. Hence these terms contribute
$(r-i)P_{k,r}^{(i)}$. If the odd component is in position $i+1$, then deleting one of the $i$ even components to its left shifts the odd component back to position $i$. Hence these terms contribute $iP_{k,r}^{(i+1)}$.

For the merge terms, fix an index contributing to $P_{k,r-1}^{(i)}$ and write it as $$(2a_1,\ldots,2a_{i-1},2a_i+1,2a_{i+1},\ldots,2a_{r-1}).$$
From an even component $2a_j$, one can extract a positive even integer $2l$ in $a_j-1$ ways, while from the odd component $2a_i+1$, one can extract it in $a_i$ ways. Hence the total multiplicity of the contributions to $P_{k,r-1}^{(i)}$ is $$a_i+\sum_{j\ne i}(a_j-1)=\frac{k-1}{2}-(r-2)=\frac{k-2r+3}{2}.$$ 
Collecting the insertion and merge terms in \eqref{eq:zero-product} gives \eqref{eq:posrec}.
\qed

Replacing $r$ by $r+1$ in \eqref{eq:posrec}, we obtain 
\begin{align}
P_{k,r}^{(i)}=-\frac{2}{k-2r+1}\left((r+1-i)P_{k,r+1}^{(i)}+iP_{k,r+1}^{(i+1)}\right)\label{eq:downward}
\end{align}
for $1\le r\le\frac{k-1}{2}$ and $1\le i\le r$. Thus, the recurrence relation \eqref{eq:downward} allows us to determine the positional parity-restricted sums at depth $r$ from those at depth $r+1$.

Fix an odd weight $k$ and set $$R=\frac{k+1}{2},$$ which is the maximal depth. At depth $R$, every even component must be $2$, while the unique odd component must be $1$. Hence, for $1\le i\le R$, we have
\begin{align}
P_{k,R}^{(i)}=\zeta_{\A}^{(2)}(\{2\}^{i-1},1,\{2\}^{R-i})=\frac{(-1)^{R-1}}{4^{R-1}}\binom{k}{2i-1}\zeta_{\A}^{(2)}(k),\label{eq:extremal}
\end{align}
where the last equality follows from \cite[Theorem 5.4]{HPT2014}.

Motivated by the downward recurrence relation \eqref{eq:downward},  we define the rational numbers $C_{k,r,i}$ recursively as follows. At the maximal depth $R$, set
\begin{align}
C_{k,R,i}=\frac{(-1)^{R-1}}{4^{R-1}}\binom{k}{2i-1},\qquad 1\le i\le R.\label{eq:Cterminal}
\end{align} 
For $1\le r<R$ and $1\le i\le r$, define 
\begin{align}
C_{k,r,i}=-\frac{2}{k-2r+1}\left((r+1-i)C_{k,r+1,i}+iC_{k,r+1,i+1}\right).\label{eq:Cdownward}
\end{align}
Since $k-2r+1\ne0$ for $r<R$, the values at depth $R$ uniquely determine all the rational numbers $C_{k,r,i}$ by downward recurrence. The following lemma gives an explicit closed form for $C_{k,r,i}$.
\begin{lem}\label{lem:Cexplicit}
Let $k$ be an odd integer and $r,i$ be positive integers satisfying $k\ge2r-1$ and $1\le i\le r$. Then the rational numbers $C_{k,r,i}$ determined by \eqref{eq:Cterminal} and \eqref{eq:Cdownward} are given explicitly by
\begin{align}
C_{k,r,i}=\frac{(-1)^{r-1}}{4^{r-1}}\binom{k-r}{r-1}\binom{r-1}{i-1}\prod_{j=0}^{s_i-1}\frac{k-2j}{k-2r+2+2j},\label{eq:Cproduct}
\end{align}
where $s_i=\min\{i,r-i\}$ and an empty product is understood to be $1$.
\end{lem}
\proof
It suffices to verify that the right-hand side of \eqref{eq:Cproduct} satisfies both the prescribed terminal values \eqref{eq:Cterminal} and the recurrence relation \eqref{eq:Cdownward}. First let $r=R$. Since $k=2R-1$, \eqref{eq:Cproduct} becomes
\begin{align*}
C_{k,R,i}=\frac{(-1)^{R-1}}{4^{R-1}}\binom{R-1}{i-1}\prod_{j=0}^{s_i-1}\frac{2R-1-2j}{2j+1},\quad s_i=\min\{i,R-i\}.
\end{align*}
A direct simplification gives
\begin{align*}
\binom{R-1}{i-1}\prod_{j=0}^{s_i-1}\frac{2R-1-2j}{2j+1}=\binom{2R-1}{2i-1}.%\label{eq:doublefact}
\end{align*}
Hence, we have $$C_{k,R,i}=\frac{(-1)^{R-1}}{4^{R-1}}\binom{k}{2i-1},$$
which is precisely the terminal value \eqref{eq:Cterminal}.

 Replacing $r$ by $r-1$ in \eqref{eq:Cdownward} and rearranging, it remains to verify
\begin{align*}
(r-i)C_{k,r,i}+iC_{k,r,i+1}+\frac{k-2r+3}{2}C_{k,r-1,i}=0%\label{eq:Crec}
\end{align*}
for $2\le r\le R$ and $1\le i\le r-1$.
Set 
\begin{align}
B_r=\frac{(-1)^{r-1}}{4^{r-1}}\binom{k-r}{r-1},\qquad\Phi_{r,s}=\prod_{j=0}^{s-1}\frac{k-2j}{k-2r+2+2j}.\label{eq:phi-rs}
\end{align}
Then, by \eqref{eq:Cproduct}, we have $$C_{k,r,i}=B_r\binom{r-1}{i-1}\Phi_{r,s_i},\quad s_i=\min\{i,r-i\}.$$

Using $$(r-i)\binom{r-1}{i-1}=i\binom{r-1}{i}=(r-1)\binom{r-2}{i-1}$$
and $$\frac{B_{r-1}}{B_r}=\frac{-4(r-1)(k-r+1)}{(k-2r+3)(k-2r+2)},$$
it suffices to prove
\begin{align}
\Phi_{r,s_i}+\Phi_{r,s_{i+1}}=\frac{k-r+1}{\frac{k}{2}-r+1}\Phi_{r-1,s_i'},\qquad s_i'=\min\{i,r-1-i\}.\label{eq:Phirec}
\end{align}

Put $k'=k/2$. We first note that, for every nonnegative integer $s$,
\begin{align}
\Phi_{r,s}+\Phi_{r,s+1}=\frac{k-r+1}{k'-r+1}\Phi_{r-1,s}.\label{eq:Phis-identity}
\end{align}
Indeed, by the definition of $\Phi_{r,s}$ in \eqref{eq:phi-rs}, we obtain
\begin{align*}
\frac{\Phi_{r,s+1}}{\Phi_{r,s}}=\frac{k'-s}{k'-r+s+1},\qquad
\frac{\Phi_{r-1,s}}{\Phi_{r,s}}=\frac{k'-r+1}{k'-r+s+1},
\end{align*}
from which \eqref{eq:Phis-identity} follows immediately.

Suppose first that $s_i\ne s_{i+1}$. By the definitions of $s_i$, $s_{i+1}$ and $s_i'$, the integers $s_i$ and $s_{i+1}$ are consecutive, with the smaller one equal to $s_i'$. Hence, we have
\begin{align*}
\Phi_{r,s_i}+\Phi_{r,s_{i+1}}=\Phi_{r,s_i'}+\Phi_{r,s_i'+1}.
\end{align*}
Then, applying \eqref{eq:Phis-identity} with $s=s_i'$ yields \eqref{eq:Phirec} immediately.

 It remains to consider $s_i=s_{i+1}$. This case occurs precisely when $r$ is odd and $i=(r-1)/2$. Set $q=(r-1)/2$. Then, we have $s_i=s_{i+1}=s_i'=q$ and
\begin{align*}
\frac{\Phi_{r-1,q}}{\Phi_{r,q}}=\frac{k'-r+1}{k'-q}.
\end{align*}
Since $k-r+1=2(k'-q)$, we have
\begin{align*}
\frac{k-r+1}{k'-r+1}\Phi_{r-1,q}=2\Phi_{r,q},%=\Phi_{r,s_i}+\Phi_{r,s_{i+1}},
\end{align*}
which proves \eqref{eq:Phirec} in the remaining case.

Therefore, \eqref{eq:Cproduct} satisfies the recurrence relation \eqref{eq:Cdownward}. Together with the verification of the terminal values, this completes the proof.
\qed
 
\begin{proofof}{Theorem \ref{thm:parity-sum}}
At the maximal depth $R=(k+1)/2$, \eqref{eq:mainparity} follows from the evaluation \eqref{eq:extremal} and the terminal value \eqref{eq:Cterminal}. Suppose that  \eqref{eq:mainparity} holds at depth $r+1$ for all $1\le j\le r+1$, with $1\le r<R$. For $1\le i\le r$, by \eqref{eq:downward}, we have
 \begin{align*}
 P_{k,r}^{(i)}&=-\frac{2}{k-2r+1}\left((r+1-i)P_{k,r+1}^{(i)}+iP_{k,r+1}^{(i+1)}\right)\\
 &=-\frac{2}{k-2r+1}\left((r+1-i)C_{k,r+1,i}+iC_{k,r+1,i+1}\right)\zeta_{\A}^{(2)}(k)\\
 &=C_{k,r,i}\zeta_{\A}^{(2)}(k),
 \end{align*}
where the last equality follows from \eqref{eq:Cdownward}. Thus, by downward induction on $r$, \eqref{eq:mainparity} holds for all $1\le r\le R$ and $1\le i\le r$. The explicit expression for $C_{k,r,i}$ is given by \eqref{eq:Cproduct}. This completes the proof.
\end{proofof}

Taking $i=r$ in Theorem \ref{thm:parity-sum}, we obtain the following corollary.

\begin{cor}
Let $r$ be a positive integer and let $k\ge2r-1$ be odd. Then we have
\begin{align*}
 \sum_{\substack{k_1+\cdots+k_r=k\\ k_1,\ldots,k_{r-1}:\ {\rm even},\ k_r:\ {\rm odd}}}\zeta_{\A}^{(2)}(k_1,\ldots,k_r)=\frac{(-1)^{r-1}}{4^{r-1}}\binom{k-r}{r-1}\zeta_{\A}^{(2)}(k).
 %\label{eq:boundary}
\end{align*}
\end{cor}

For example, when $r=2$ and $k=5$, we obtain
\begin{align*}
\zeta_{\A}^{(2)}(2,3)+\zeta_{\A}^{(2)}(4,1)=-\frac{3}{4}\zeta_{\A}^{(2)}(5).
\end{align*}
When $r=3$ and $k=7$, we obtain
\begin{align*}
\zeta_{\A}^{(2)}(2,2,3)+\zeta_{\A}^{(2)}(2,4,1)+\zeta_{\A}^{(2)}(4,2,1)=\frac{3}{8}\zeta_{\A}^{(2)}(7).
\end{align*}

%%---------------------------%%--------------------
\section{Ohno--Zagier type relation}\label{sec:oz-sum}
Recall that 
\begin{align*}
S(k,r,h)=\sum_{\boldsymbol{k}\in I(k,r,h)}\zeta_{\A}^{(2)}(k_1,\ldots,k_r),\quad k\ge r\ge 1,\quad 0\le h\le r,
\end{align*}
where $I(k, r, h)$ denotes the set of indices of weight $k$, depth $r$, and height $h$.
For formal variables $x,y,u$, we consider the generating function
\begin{align*}
F(x,y,u)=1+\sum_{k\ge r\ge1}\sum_{h=0}^{r}S(k,r,h)x^{k-r}y^ru^h.%\label{eq:F}
\end{align*}

In this section, we first prove Theorem \ref{thm:oz-sum}, which gives an Ohno--Zagier type identity for the above generating function, and then derive some applications.

\subsection{Proof of Theorem \ref{thm:oz-sum}}\label{subsec:proofTheorem}
For an odd prime $p$ and an index $\boldsymbol{k}=(k_1,\ldots,k_r)$, define 
\begin{align}
\zeta_{p}^{(2)}(\boldsymbol{k})=\zeta_{p}^{(2)}(k_1,\ldots,k_r)=\sum_{0<n_1<\cdots<n_r<p/2}\frac{1}{n_1^{k_1}\cdots n_r^{k_r}}\pmod p.\label{eq:pzeta}
\end{align}
Then, we have $$\zeta_{\A}^{(2)}(\boldsymbol{k})=(\zeta_{p}^{(2)}(\boldsymbol{k}))_p.$$
Accordingly, set
\begin{align*}
S_p(k,r,h)=\sum_{\boldsymbol{k}\in I(k,r,h)}\zeta_{p}^{(2)}(k_1,\ldots,k_r),\quad k\ge r\ge 1,\quad 0\le h\le r,
\end{align*}
and define
\begin{align*}
 F_p(x,y,u)=1+\sum_{k\ge r\ge1}\sum_{h=0}^{r}S_p(k,r,h)x^{k-r}y^ru^h.%\label{eq:F}
\end{align*}
Note that $x$ records the weight minus the depth, $y$ records the depth, and $u$ records the height. Moreover,
$$F(x,y,u)=(F_p(x,y,u))_p\quad {\rm in}~\A[[x,y,u]].$$

We first establish the following product formula for $F_p(x,y,u)$.
\begin{lem}\label{lem:Fp-product}
Let $p$ be an odd prime and set $N=(p-1)/2$. Then
\begin{align}
 F_p(x,y,u)=\prod_{m=1}^{N}\frac{(m-\alpha)(m-\beta)}{m(m-x)},\label{eq:Fp-product}
\end{align}
where $\alpha$ and $\beta$ are formally determined by 
\begin{align}
\alpha+\beta=x-y,\qquad\alpha\beta=-xy(1-u).\label{eq:alphabeta2}
\end{align}
\end{lem}
\proof
For a fixed index $\boldsymbol{k}=(k_1,\ldots,k_r)$ of weight $k$ and height $h$,  each summand
$$\frac{x^{k-r}y^ru^h}{n_1^{k_1}\cdots n_r^{k_r}},
\qquad 1\le n_1<\cdots<n_r\le N,$$
corresponds to choosing $r$ distinct integers from $\{1,\ldots,N\}$, where $n_i$ carries the exponent $k_i$. For each $m\in\{1,\ldots,N\}$, either $m$ is not chosen or it is chosen with some exponent $j\ge1$. Since the summation variables are strictly increasing, each $m$ can occur at most once.

If $m$ is not chosen, its contribution is $1$. If $m$ is  chosen with exponent $j\ge 1$, then the factor $y$ records the increase of the depth by $1$ and  $x^{j-1}$ records the contribution $j-1$ to $k-r$. Moreover, an additional factor $u$ occurs precisely when $j\ge2$, since in this case the height increases by one. Hence, the contribution corresponding to $m$ is $y/m$ when $j=1$, and $yu x^{j-1}/m^j$ when $j\ge2$. Therefore, multiplying over all $m=1,\ldots,N$, we obtain
\begin{align*}
 F_p(x,y,u)=\prod_{m=1}^{N}\left(1+\frac{y}{m}+yu\sum_{j\ge2}\frac{x^{j-1}}{m^j}\right).
\end{align*}
Since $$\sum_{j\ge2}\frac{x^{j-1}}{m^j}=\frac{x}{m(m-x)},$$
we have $$1+\frac{y}{m}+yu\sum_{j\ge2}\frac{x^{j-1}}{m^j}=\frac{m^2+(y-x)m-xy(1-u)}{m(m-x)}.$$

Using \eqref{eq:alphabeta2}, we have
\begin{align*}
m^2+(y-x)m-xy(1-u)=(m-\alpha)(m-\beta).
\end{align*}
Hence, we get the desired formula \eqref{eq:Fp-product}.
\qed

\begin{proofof}{Theorem \ref{thm:oz-sum}}
By Lemma \ref{lem:Fp-product}, for any odd prime $p$,
\begin{align*}
 F_p(x,y,u)=\prod_{m=1}^{N}\frac{(1-\alpha/m)(1-\beta/m)}{1-x/m},
\end{align*}
where $N=(p-1)/2$ and $\alpha,\beta$ are determined by \eqref{eq:alphabeta2}. For any fixed coefficient in $x,y,u$, the following computation is valid for all sufficiently large $p$. Taking the formal logarithm and using $$\log(1-z)=-\sum\limits_{n\ge1}\frac{z^n}{n},$$ we obtain
\begin{align*}
\log F_p(x,y,u)&=\sum_{m=1}^{N}\left(\log\left(1-\frac{\alpha}{m}\right)+\log\left(1-\frac{\beta}{m}\right)-\log\left(1-\frac{x}{m}\right)\right)\\
&=\sum_{n\ge1}\frac{x^n-\alpha^n-\beta^n}{n}\sum_{m=1}^{N}\frac{1}{m^n}\\
&=\sum_{n\ge1}\frac{\zeta_p^{(2)}(n)}{n}\left(x^n-\alpha^n-\beta^n\right).
\end{align*}
Since $\zeta_{\A}^{(2)}(n)=(\zeta_p^{(2)}(n))_p$ and $F(x,y,u)=(F_p(x,y,u))_p$, we have
\begin{align*}
\log F(x,y,u)=\sum_{n\ge1}\frac{\zeta_{\A}^{(2)}(n)}{n}\left(x^n-\alpha^n-\beta^n\right).
\end{align*}
Exponentiating both sides yields  \eqref{eq:OZ} as desired.
\end{proofof}

\subsection{Sum of fixed weight and depth}
Let
\begin{align*}
S(k,r)=\sum_{k_1+\cdots+k_r=k}\zeta_{\mathcal A}^{(2)}(k_1,\ldots,k_r)
\end{align*}
denote the sum of level-two FMZVs of fixed weight $k$ and depth $r$ with $k\ge r\ge1$.

Taking $u=1$ in Theorem \ref{thm:oz-sum}, we have $\alpha\beta=0$ and $\alpha+\beta=x-y$. We may take $(\alpha,\beta)=(0,x-y)$. Since $$\sum\limits_{h=0}^rS(k,r,h)=S(k,r),$$ we obtain the following formula for the generating function of $S(k,r)$.

\begin{cor}
Let $x,y$ be formal variables. In $\A[[x,y]]$, we have
\begin{align}
1+\sum_{k\ge r\ge1}S(k,r)x^{k-r}y^r=\exp\left\{\sum_{n\ge1}\frac{\zeta_{\mathcal A}^{(2)}(n)}{n}\left(x^n-(x-y)^n\right)\right\}.\label{eq:fixed-wtdep}
\end{align}
\end{cor}

Since $\zeta_{\mathcal A}^{(2)}(n)=0$ for even $n$ by \eqref{eq:evenzero}, the right-hand side of \eqref{eq:fixed-wtdep} is invariant under the substitution $x\mapsto y-x$. Therefore, we have
\begin{align*}
\sum_{k\ge r\ge1}S(k,r)x^{k-r}y^r=\sum_{k\ge r\ge1}S(k,r)(y-x)^{k-r}y^r.
\end{align*}
Comparing the coefficients of $x^{k-r}y^r$ on both sides of the above formula, we obtain
\begin{align}
S(k,r)=(-1)^{k-r}\sum_{j=1}^{r}\binom{k-j}{r-j}S(k,j).\label{eq:fixed-downdep}
\end{align}
More precisely, we have
\begin{align*}
S(k,r)=-\frac{1}{2}\sum_{j=1}^{r-1}\binom{k-j}{r-j}S(k,j),\quad k-r:\ {\rm odd},
\end{align*}
and
\begin{align*}
\sum_{j=1}^{r-1}\binom{k-j}{r-j}S(k,j)=0,\quad k-r:\ {\rm even}.
\end{align*}

\begin{rem}
When $k-r$ is odd, \eqref{eq:fixed-downdep} expresses
$S(k,r)$ in terms of sums of smaller depth, while when $k-r$ is even, it gives a relation among sums of smaller depth. On the other hand, \cite[Theorem 3.1 (i)]{KanekoMurakamiYoshihara2023} expresses $S(k,r)$ as a linear combination of level-two FMZVs whose index components are all odd.
\end{rem}

\subsection{Sum of height zero and height one}
An index has height zero if and only if all its components are equal to $1$. Here and below, we use the convention $\zeta_{\A}^{(2)}(\{1\}^0)=1$. Taking $u=0$ in Theorem \ref{thm:oz-sum}, we have $\alpha+\beta=x-y$ and $\alpha\beta=-xy$, 
so that we may take $(\alpha,\beta)=(x,-y)$. Hence, using \eqref{eq:evenzero}, we obtain the following generating function for level-two FMZVs of height zero.
\begin{cor}
Let $y$ be a formal variable. Then
\begin{align}
\sum_{r\ge0}\zeta_{\A}^{(2)}(\{1\}^r)y^r=\exp\left\{\sum_{\substack{n\ge1,\ n:\ {\rm odd}}}\frac{\zeta_{\A}^{(2)}(n)}{n}y^n\right\}\quad in\ \A[[y]].\label{eq:height-zero}
\end{align}
\end{cor}

Differentiating both sides of \eqref{eq:height-zero} with respect to $y$ and comparing
coefficients of $y^{r-1}$ gives
\begin{align*}
r\zeta_{\A}^{(2)}(\{1\}^r)=\sum_{\substack{1\le n\le r\\ n:\ {\rm odd}}}\zeta_{\A}^{(2)}(n)\zeta_{\A}^{(2)}(\{1\}^{r-n}),\quad r\ge1.
\end{align*}
Thus $\zeta_{\A}^{(2)}(\{1\}^r)$ can be determined recursively from the depth-one level-two FMZVs of odd weight. For example,
\begin{align*}
\zeta_{\A}^{(2)}(1,1)&=\frac{1}{2}\bigl(\zeta_{\A}^{(2)}(1)\bigr)^2,\\
\zeta_{\A}^{(2)}(1,1,1)&=\frac{1}{6}\bigl(\zeta_{\A}^{(2)}(1)\bigr)^3+\frac{1}{3}\zeta_{\A}^{(2)}(3),\\
\zeta_{\A}^{(2)}(1,1,1,1)&=\frac{1}{24}\bigl(\zeta_{\A}^{(2)}(1)\bigr)^4+\frac{1}{3}\zeta_{\A}^{(2)}(1)\zeta_{\A}^{(2)}(3).
\end{align*}

\begin{rem}
More generally, Zhao \cite[Proposition 4]{ZhaoFoundations2024} states explicit evaluations for $\zeta^{(2)}_\A(k,k)$ and $\zeta^{(2)}_\A(k,k,k)$ for $k\ge 1$. For arbitrary depth $r$, the same proposition shows that $\zeta^{(2)}_\A(\{k\}^r)$ can be expressed as a $\mathbb{Q}$-linear combination of products of depth-one level-two FMZVs of odd weight.
\end{rem}

We next consider the sum of  level-two FMZVs of height one. Since an index of weight $k$,
depth $r$, and height one is necessarily of the form
\begin{align*}
(\{1\}^a,k-r+1,\{1\}^{r-1-a}),
\qquad 1\le r\le k-1,\quad
0\le a\le r-1,
\end{align*}
we have
\begin{align*}
S(k,r,1)=\sum_{a=0}^{r-1}\zeta_{\A}^{(2)}(\{1\}^a,k-r+1,\{1\}^{r-1-a}).
\end{align*}

By extracting the coefficient of $u$ in \eqref{eq:OZ}, we obtain the following sum formula for level-two FMZVs of height one.

\begin{cor}\label{cor:height-one}
For positive integers $k,r$ with $k\ge r+1$, we have
\begin{align}
S(k,r,1)=\sum_{j=0}^{r-1}(-1)^j\zeta_{\A}^{(2)}(k-r+j+1)\zeta_{\A}^{(2)}(\{1\}^{r-j-1}).\label{eq:height-one}
\end{align}
\end{cor}
\proof
Differentiating both sides of \eqref{eq:alphabeta} with respect to $u$ gives
$$\frac{\partial\alpha}{\partial u}+\frac{\partial\beta}{\partial u}=0,\quad \frac{\partial\alpha}{\partial u}\beta+\alpha\frac{\partial\beta}{\partial u}=xy.$$
At $u=0$, we may take $(\alpha,\beta)=(x,-y)$. Hence, we have 
\begin{align}
\left.\frac{\partial}{\partial u}\left(\alpha^n+\beta^n\right)\right|_{u=0}=-nxy\sum\limits_{j=0}^{n-2}(-1)^jx^{n-2-j}y^j.\label{eq:derivation-alphabeta}
\end{align}
Then, differentiating both sides of \eqref{eq:OZ} with respect to $u$ and using \eqref{eq:height-zero} and \eqref{eq:derivation-alphabeta}, we obtain
\begin{align}
&\sum\limits_{r\ge1}\sum_{k\ge r+1}S(k,r,1)x^{k-r}y^r\notag\\
&=\left(\sum_{d\ge0}\zeta_{\A}^{(2)}(\{1\}^d)y^d\right)\left(\sum_{n\ge2}\zeta_{\A}^{(2)}(n)\sum_{j=0}^{n-2}(-1)^j x^{n-1-j}y^{j+1}\right).\label{eq:gen-height-one}
\end{align}
Comparing the coefficients of $x^{k-r}y^r$ on both sides of \eqref{eq:gen-height-one} yields \eqref{eq:height-one} as desired.
\qed

\eqref{eq:height-one} shows that the sum of level-two FMZVs of height one can be expressed in terms of depth-one values and the
height-zero values $\zeta_{\A}^{(2)}(\{1\}^{j})$. Since the values $\zeta_{\A}^{(2)}(\{1\}^{j})$ are themselves determined
recursively by depth-one values, every sum $S(k,r,1)$ can ultimately
be expressed as a polynomial in depth-one level-two FMZVs of odd weight.

For example, when $r=3$, we have
\begin{align*}
&\zeta_{\A}^{(2)}(k-2,1,1)+\zeta_{\A}^{(2)}(1,k-2,1)+\zeta_{\A}^{(2)}(1,1,k-2)\\
&\qquad=\frac{1}{2}\zeta_{\A}^{(2)}(k-2)\bigl(\zeta_{\A}^{(2)}(1)\bigr)^2-\zeta_{\A}^{(2)}(k-1)\zeta_{\A}^{(2)}(1)+\zeta_{\A}^{(2)}(k).
\end{align*}
In particular, using \eqref{eq:evenzero}, the cases $k=4$ and $k=5$
give
\begin{align*}
\zeta_{\A}^{(2)}(2,1,1)+\zeta_{\A}^{(2)}(1,2,1)+\zeta_{\A}^{(2)}(1,1,2)=-\zeta_{\A}^{(2)}(3)\zeta_{\A}^{(2)}(1),
\end{align*}
and
\begin{align*}
\zeta_{\A}^{(2)}(3,1,1)+\zeta_{\A}^{(2)}(1,3,1)+\zeta_{\A}^{(2)}(1,1,3)=\frac{1}{2}\zeta_{\A}^{(2)}(3)\bigl(\zeta_{\A}^{(2)}(1)\bigr)^2+\zeta_{\A}^{(2)}(5),
\end{align*}
respectively.

\begin{rem}
Corollary \ref{cor:height-one} can also be derived directly from the
harmonic product.
\end{rem}

\subsection{Sum of maximal height}
We now consider the maximal-height case $h=r$, which requires every component of the index to be at least $2$. For $r\ge1$ and $k\ge 2r$, define the sum of level-two FMZVs of maximal height
\begin{align*}
M(k,r)=\sum_{\substack{k_1+\cdots+k_r=k\\ k_j\ge2\ (1\le j\le r)}}\zeta_{\A}^{(2)}(k_1,\ldots,k_r).
\end{align*}

We obtain the following generating function for $M(k,r)$.
\begin{cor}
Let $x,y$ be formal variables. In $\A[[x,y]]$, we have
\begin{align}
1+\sum_{r\ge1}\sum_{k\ge2r}M(k,r)x^{k-2r}y^r=\exp\left\{\sum_{n\ge1}\frac{\zeta_{\A}^{(2)}(n)}{n}\left(x^n-\gamma^n-\delta^n\right)\right\},\label{eq:maxheight}
\end{align}
where $\gamma$ and $\delta$ are formally determined by
\begin{align}
\gamma+\delta=x,\qquad\gamma\delta=y.\label{eq:gammadelta}
\end{align}
\end{cor}
\proof
For each odd prime $p$, put $N=(p-1)/2$. As in the proof of Lemma \ref{lem:Fp-product}, for each $m\in\{1,\ldots,N\}$, either $m$ is not chosen, or it is chosen with an exponent $j\ge2$. Hence, we have
\begin{align}
1+\sum_{r\ge1}\sum_{k\ge2r}M(k,r)x^{k-2r}y^r&=\left(\prod_{m=1}^{N}\left(1+\sum_{j\ge2}\frac{yx^{j-2}}{m^j}\right)\right)_p\notag\\&=\left(\prod_{m=1}^{N}\frac{(m-\gamma)(m-\delta)}{m(m-x)}\right)_p,\label{eq:Mproduct}
\end{align}
where $\gamma$ and $\delta$ satisfy \eqref{eq:gammadelta}. For any fixed coefficient in \(x,y\), the same formal logarithmic argument as in the proof of Theorem \ref{thm:oz-sum} is valid for all sufficiently large $p$, and thus gives \eqref{eq:maxheight}.
\qed

Denote by $G(x,y)$ the generating function 
\begin{align*}
G(x,y)=1+\sum_{r\ge1}\sum_{k\ge2r}
M(k,r)x^{k-2r}y^r.
\end{align*}
By \eqref{eq:gammadelta}, we have 
\begin{align*}
  \frac{(m-\gamma)(m-\delta)}{m(m-x)}=1+\frac{y}{m(m-x)}.
\end{align*}
Then, taking the formal logarithm of \eqref{eq:Mproduct} and using
\begin{align*}
\log(1+z)&=\sum_{q\ge1}\frac{(-1)^{q-1}}{q}z^q,\quad\frac{1}{(m-x)^q}=\sum_{a\ge0}\binom{q+a-1}{a}\frac{x^a}{m^{q+a}},
\end{align*}
we obtain
\begin{align}
\log G(x,y)=\sum_{q\ge1}\sum_{a\ge0}\frac{(-1)^{q-1}}{q}\binom{q+a-1}{a}\zeta_{\A}^{(2)}(2q+a)x^ay^q.\label{eq:logM}
\end{align}

By \eqref{eq:evenzero}, only odd $a$ contribute to
\eqref{eq:logM}. Hence the coefficient of $xy^r$ in
the right-hand side of \eqref{eq:logM} is $(-1)^{r-1}\zeta_{\A}^{(2)}(2r+1)$. Since $\log G(x,y)$ is divisible by $x$, all terms $(\log G(x,y))^s$ with $s\ge2$ have degree at least $2$ in $x$. Hence the coefficient of \(xy^r\) in \(G(x,y)=\exp(\log G(x,y))\) is the same as that in \(\log G(x,y)\). Moreover, an index of weight $2r+1$ and depth $r$, with every component at least $2$, must contain exactly one component equal to $3$, while all the remaining components are equal to $2$. Therefore, we have
\begin{align}
M(2r+1,r)=\sum_{j=0}^{r-1}\zeta_{\A}^{(2)}(\{2\}^{j},3,\{2\}^{r-1-j})=(-1)^{r-1}\zeta_{\A}^{(2)}(2r+1).\label{eq:Mminimal1}
\end{align}

Furthermore, $\log G(x,y)$ has no term of degree $2$ in $x$, so the coefficient of $x^2y^r$ in $G(x,y)=\exp(\log G(x,y))$ comes only from one half of the square of the $x$-linear part of $\log G(x,y)$.
By \eqref{eq:logM}, this $x$-linear part is
\begin{align*}
x\sum_{q\ge1}
(-1)^{q-1}\zeta_{\A}^{(2)}(2q+1)y^q.
\end{align*}
Hence,
\begin{align*}
M(2r+2,r)=\frac{(-1)^r}{2}\sum_{j=1}^{r-1}\zeta_{\A}^{(2)}(2j+1)\zeta_{\A}^{(2)}(2r-2j+1).%\label{eq:Mminimal2}
\end{align*}

The following result generalizes \eqref{eq:Mminimal1}.

\begin{prop}\label{prop:insertion}
For integers $q\ge0$ and $m\ge1$, we have
\begin{align}
\sum_{\substack{a+b=q\\ a,b\ge 0}}
\zeta_{\A}^{(2)}(\{2\}^{a},m,\{2\}^{b})=(-1)^q\zeta_{\A}^{(2)}(m+2q).\label{eq:insertion}
\end{align}
\end{prop}
\proof
For each odd prime $p$, put $N=(p-1)/2$. We consider the generating
function
\begin{align*}
H_p(y)=\sum_{q\ge0}\sum_{\substack{a+b=q\\ a,b\ge0}}\zeta_p^{(2)}(\{2\}^{a},m,\{2\}^{b})y^q,
\end{align*}
where $\zeta_p^{(2)}$ is defined by \eqref{eq:pzeta}. Then we have
\begin{align}
H_p(y)=\sum_{n=1}^{N}\frac{1}{n^m}\prod_{\substack{1\le l\le N\\ l\ne n}}\left(1+\frac{y}{l^2}\right).
\label{eq:Hp}
\end{align}
Indeed, after fixing $n$ as the variable carrying the exponent $m$,
the selected variables smaller than $n$ contribute the block
$\{2\}^{a}$, while those larger than $n$ contribute the block
$\{2\}^{b}$. Now define
\begin{align}
B_p(y)=\prod_{l=1}^{N}\left(1+\frac{y}{l^2}\right).\label{eq:Bpdef}
\end{align}
For any fixed coefficient in $y$, the following computation is valid for all sufficiently large $p$. Taking the formal logarithm, we obtain
\begin{align*}
\log B_p(y)=\sum_{n=1}^{N}
\log\left(1+\frac{y}{n^2}\right)=
\sum_{j\ge1}\frac{(-1)^{j-1}}{j}y^j\sum_{n=1}^{N}\frac{1}{n^{2j}}.
\end{align*}
By \eqref{eq:evenzero}, for every $j\ge1$,
\begin{align*}
\left(\sum_{n=1}^{N}\frac{1}{n^{2j}}\right)_p=\zeta_{\A}^{(2)}(2j)=0.
\end{align*}
Hence, we have
\begin{align}
\bigl(B_p(y)\bigr)_p=1.\label{eq:By}
\end{align}
Using \eqref{eq:Hp}, \eqref{eq:Bpdef} and \eqref{eq:By}, we get
\begin{align*}
\sum_{q\ge0}\sum_{\substack{a+b=q\\ a,b\ge0}}\zeta_{\A}^{(2)}(\{2\}^{a},m,\{2\}^{b})y^q=\left(\sum_{n=1}^{N}\frac{1}{n^m}\frac{1}{1+y/n^2}\right)_p.
\end{align*}
Since
\begin{align*}
\frac{1}{1+y/n^2}=\sum_{q\ge0}(-1)^q\frac{y^q}{n^{2q}},
\end{align*}
it follows that
\begin{align*}
\sum_{q\ge0}\sum_{\substack{a+b=q\\ a,b\ge0}}\zeta_{\A}^{(2)}(\{2\}^{a},m,\{2\}^{b})y^q=\sum_{q\ge0}(-1)^q\zeta_{\A}^{(2)}(m+2q)y^q.
\end{align*}
Comparing the coefficients of $y^q$ on both sides, we obtain the desired formula \eqref{eq:insertion}.
\qed

Taking $m=3$ and $q=r-1$ in Proposition \ref{prop:insertion} recovers \eqref{eq:Mminimal1}. On the other hand, taking $m=1$ gives
\begin{align}
\sum_{j=0}^{r-1}\zeta_{\A}^{(2)}(\{2\}^{j},1,\{2\}^{r-1-j})=(-1)^{r-1}\zeta_{\A}^{(2)}(2r-1).\label{eq:sumextremal}
\end{align}
\eqref{eq:sumextremal} gives the sum over all possible positions of $1$ in the indices
$(\{2\}^{a},1,\{2\}^{b})$ and \eqref{eq:extremal} refines \eqref{eq:sumextremal} by evaluating
each position separately.

\begin{rem}
In fact, by the harmonic product, we have
\begin{align*}
q\zeta_{\A}^{(2)}(\{2\}^q)=\sum_{j=1}^{q}(-1)^{j-1}\zeta_{\A}^{(2)}(2j)\zeta_{\A}^{(2)}(\{2\}^{q-j}).
\end{align*}
Since $\zeta_{\A}^{(2)}(2j)=0$ for $j\ge1$, this implies
\begin{align}
\zeta_{\A}^{(2)}(\{2\}^q)=0,\qquad q\ge1.\label{eq:repp}
\end{align}
Moreover, \eqref{eq:insertion} can also be derived directly from the harmonic product together with \eqref{eq:repp}.
\end{rem}

%\section*{Acknowledgments}

\end{document}